\documentclass[a4paper]{amsart}

\usepackage[all]{xy}

\usepackage{amsmath}

\usepackage{amssymb}

\usepackage{latexsym}

\usepackage{amsthm}

\usepackage{mathrsfs}

\usepackage{float}

\usepackage{geometry}

\usepackage{mathdots}

\newcommand{\A}{\mathbb{A}}

\newcommand{\C}{\mathbb{C}}

\newcommand{\D}{\mathbb{D}}

\newcommand{\disc}{\mathfrak{d}}

\newcommand{\F}{\mathbb{F}}

\newcommand{\g}{\mathfrak{g}}

\renewcommand{\H}{\mathbb{H}}

\renewcommand{\k}{\mathfrak{k}}

\newcommand{\M}{\operatorname{M}}

\newcommand{\cO}{\mathcal{O}}

\newcommand{\p}{\mathfrak{p}}

\newcommand{\R}{\mathbb{R}}

\newcommand{\s}{\mathfrak{s}}

\renewcommand{\u}{\mathfrak{u}}

\newcommand{\V}{\mathcal{V}}

\newcommand{\W}{\mathcal{W}}

\newcommand{\Z}{\mathbb{Z}}

\newcommand{\Hom}{\operatorname{Hom}}

\newcommand{\GL}{\operatorname{GL}}

\newcommand{\Ind}{\operatorname{Ind}}

\newcommand{\Isom}{\operatorname{Isom}}

\newcommand{\ord}{\operatorname{ord}}

\newcommand{\sgn}{\operatorname{sgn}}

\newcommand{\diag}{\operatorname{diag}}

\newcommand{\Sp}{\operatorname{Sp}}

\newcommand{\std}{\operatorname{std}}

\newcommand{\End}{\operatorname{End}}

\newcommand{\Vol}{\operatorname{Vol}}

\newcommand{\tru}{\bigtriangleup}

\newcommand{\trd}{\bigtriangledown}

\newtheorem{df}{Definition}[section]

\newtheorem{thm}[df]{Theorem}

\newtheorem{prop}[df]{Proposition}

\newtheorem{lem}[df]{Lemma}

\newtheorem{rem}[df]{Remark}

\makeatletter
    
    \@addtoreset{equation}{section}
\makeatother

\title{On the Local Factors of Irreducible Representations of Quaternionic Unitary Groups}

\author{Hirotaka KAKUHAMA}

\email{hkaku@math.kyoto-u.ac.jp}
\address{Department of Mathematics, Kyoto University, Kitashirakawa Oiwake-cho, Sakyo-ku, Kyoto 606-8502, Japan}

\date{}

\begin{document}

\maketitle

\begin{abstract}
In this paper, we give a precise definition of the analytic $\gamma$-factor of irreducible representations of quaternionic unitary groups, which extends a work of Lapid-Rallis.  
\end{abstract}

\setcounter{tocdepth}{1}
\tableofcontents



\section{
            Introduction
            }


The doubling method of Piatetski-Shapiro and Rallis \cite{GPSR87, PSR86} is a theory of integral representation of standard $L$-functions.
Lapid-Rallis \cite{LR05} elaborated on the doubling method, and gave a definition of the analytic $\gamma$-factor of irreducible representations of general linear groups, orthogonal groups, symplectic groups and unitary groups. This was extended by Gan \cite{Gan12} to metaplectic groups.
Moreover, Yamana \cite{Yam14} established their analytic properties and verified basic properties for classical groups containing isometry groups of hermitian or skew-hermitian forms over quaternion algebras. 
However, his work was not enough to characterize the $\gamma$-factor.
The purpose of this paper is to give a precise definition of the analytic $\gamma$-factor and to characterize it in this case.


Now, we explain our result in more detail. 
Let $F$ be a local field of characteristic zero and let $D$ be a quaternion algebra over $F$.
We consider an $\epsilon$-hermitian space over $D$ (see \S\ref{hs}). Then the \emph{quaternionic unitary group} is defined as the isometry group of the $\epsilon$-hermitian space.
Let $G$ be either a general linear group $\GL_n(D)$ or a quaternionic unitary group, let $\pi$ be an irreducible representation of $G$, let $\omega$ be a character of $F^\times$, let $\psi$ be a non-trivial additive character of $F$. In this paper, we define the $\gamma$-factor of $\pi$ by
\[
\gamma^\V(s + \frac{1}{2},\pi\times\omega,\psi)
 = \Gamma^\V(s,\pi,\omega,A,\psi) c_{\pi}(-1) R(s,\omega,A,\psi)
\]
where 
\begin{itemize}
\item $A$ is some element of the Lie algebra of $G^\Box$ (for definition, see \S\ref{doubledsp});
\item $\Gamma^\V(s,\pi,\omega,A,\psi)$ is a ``normalized $\Gamma$-factor'', which is defined in \S\ref{dwf}. This factor is obtained from a functional equation of doubling zeta integrals;
\item $c_\pi$ is the central character of $\pi$;
\item $R(s,\omega,A,\psi)$ is a correction term, which is defined in \S\ref{pnf}.
\end{itemize}
We expect the $\gamma$-factor to satisfy
\begin{align}\label{desi}
\gamma^\V(s,\pi\times\omega,\psi) = \gamma(s, \phi_\pi\otimes\omega, \std, \psi)
\end{align}
if the $L$-parameter $\phi_\pi$ is attached to $\pi$. Here, we denote by $\std$ the standard representation of  the $L$-group $^LG$ into $\GL_N(\C)$.
As in \cite{LR05}, one can show that the $\gamma$-factor $\gamma^\V(s,\pi\times\omega,\psi)$ satisfies the global functional equation.
Yamana showed that the $\gamma$-factor $\gamma^\V(s,\pi\times\omega,\psi)$ satisfies some required properties: the multiplicativity, the self-duality, the (local) functional equation, and the equation \eqref{desi} for $G = \GL_n(D)$.
In this paper, we prove the equation \eqref{desi} for $G$ in the archimedean case. Moreover, we prove that the $\gamma$-factor $\gamma^\V(s,\pi\time\omega,\psi)$ is characterized by some required properties. Both are stated in Theorem \ref{maingamma}.


In the rest of the introduction, we explain the contents of this paper. In \S\ref{qug}-\S\ref{nio}, we explain the (local) framework of the doubling method. In \S\ref{sm}, we give a definition of the $\gamma$-factor and state the main theorem. We also recall the definition of the Lapid-Rallis $\gamma$-factor. We also define the $L$-factor and the $\epsilon$-factor as in \cite[\S10]{LR05}. In \S\ref{prf} and \S\ref{calcu}, we prove the main theorem. 

In \S\ref{uniqprf}, we prove that the properties of the $\gamma$-factor $\gamma(s,\pi\times\omega,\psi)$ stated in the main theorem determine it uniquely by using a global argument. 
In \S\ref{fp}, we treat some properties which come from the framework of the doubling method. For example, the multiplicativity, the functional equation and the self duality. Note that there is a minor error in the proof of the multiplicativity of \cite{LR05}. So we correct it at this opportunity. 
In \S\ref{sc}, we prove that the $\gamma$-factor defined in \S\ref{sm} is nothing other than Lapid-Rallis $\gamma$-factor in the split case.
Finally, in \S\ref{relLLC}, we discuss the desired property \eqref{desi}. In the archimedean case, we prove it as in \cite{LR05}. In the non-archimedean case, we prove it admitting the local Langlands correspondence.
In \S\ref{calcu}, we discuss some explicit calculations to complete the proof of the main theorem.

In \S\ref{appli}, we give two applications. First, we determine the local root number of irreducible representations of quaternionic unitary groups. This extends the result of Lapid-Rallis \cite[Theorem 1]{LR05}. Second, over a non-archimedean local field of odd residual characteristic, we give an explicit formula of the doubling zeta integrals of some spherical representations with respect to a certain subgroup. In the proof, we use the explicit formula of the $\gamma$-factor of the trivial representation obtained in \S\ref{calcu}.

\subsection*{
                 Acknowledgments:
                  }

The author would like to thank my supervisor A. Ichino for many advices. The author also would like to thank the referee for sincere and useful comments.


\section{
            Quaternionic unitary groups
            }\label{qug}

Let $F$ be a field of characteristic zero, and let $D$ be a quaternion algebra over $F$. Denote by $*$ the main involution of $D$. For a finite dimensional right vector space $V$ over $D$, we denote the reduced norm of $\End_D(V)$ by $N_V$. For simplicity, we denote the reduced norm of $\M_n(D)$ by $N$. 
Note that $D$ is possibly split since it appears as a localization of a division quaternion algebra over a number field. If $D$ is split, then ``finite dimensional right vector space over $D$'' means ``free right module over $D$ of finite rank''. We discuss the split case more in \S\ref{morita}.


\subsection{
                Hermitian and skew-hermitian spaces over $D$.
                }
                \label{hs}

Fix $\epsilon =\pm1$. Let $\V = (V,h)$ be a pair consisting of an $n$-dimensional right vector space $V$ over $D$ and a map $h:V \times V \rightarrow D$ such that
\begin{enumerate}
\item $h(va,wb) = a^* h(v,w)b$
\item $h(v_1+v_2,w) = h(v_1,w) + h(v_2,w)$
\item $h(w,v) = \epsilon h(v,w)^*$
\end{enumerate}
for $a,b \in D$ and $v,w,v_1,v_2\in V$. We call a pair $\V=(V,h)$ a \emph{hermitian space} (resp. a \emph{skew-hermitian space}) over $D$ if $h$ is non-degenerate and $\epsilon =1$ (resp. $h$ is non-degenerate and $\epsilon = -1$). We use the term ``$\epsilon$-hermitian space'' when we treat hermitian space and skew-hermitian space at the same time. In this paper, we consider the three cases: either the linear case (i.e. the case $h=0$), the hermitian case or the skew-hermitian case.

Let $\V$ be an $n$-dimensional $\epsilon$-hermitian space over $D$. We define the \emph{discriminant} of $\V$ by
\[
 \disc(\V) := (-1)^n N( (h(v_i,v_j) )_{ij} ) \in F^\times / F^{\times 2}
\]
where $v_1,\ldots,v_n$ is a basis of $V$. Then $\disc(\V)$ does not depend on the choice of basis.


\subsection{
                Parabolic subgroups
                }
                \label{parabsub}

Let 
\[
G= \Isom(\V)=\{ g \in \GL(V) \mid h(gv,gw) = h(v,w) \mbox{ for } v,w \in V \}
\]
be the isometry group of $\V$ where $\GL(V)$ is the general linear group of $V$. If $P$ is a maximal parabolic subgroup of $G$ over $F$, then there is a totally isotropic subspace $W$ (in the linear case, it is just a subspace) of $V$ such that $P$ coincides with the stabilizer 
\[
P(W)= \{ g \in G \mid gW = W \}
\]
of $W$. We denote the unipotent radical of $P(W)$ by $U(W)$. We put 
\[
\W_0 := (W,0) , \ \ \W_1 := (W^\bot/W,h)
\]
where 
\[
W^\bot:= \{ v \in V \mid h(v,x) = 0 \mbox{ for all } x \in W\}.
\]
Then there is the exact sequence
\[
1 \rightarrow U(W) \rightarrow P(W) \rightarrow \GL(W) \times \Isom(\W_1) \rightarrow 1,
\]
and any Levi subgroup of $P$ is isomorphic to 
\[
\GL(W) \times \Isom(\W_1).
\]
We denote the Lie algebra of $G$ (resp. $P(W), U(W)$) by $\g$ (resp. $\p(W),\u(W)$).


\subsection{
                Morita equivalence
                }\label{morita}

In this subsection we assume that $D$ is split. Then we may identify $D$ with $\M_2(F)$.
Put
\[
e:= \begin{pmatrix} 1 & 0 \\ 0 & 0 \end{pmatrix}.
\]
Then for an $n$-dimensional right vector space $V$ over $D$, $V^\natural = Ve$ is a $2n$-dimensional vector space over $F$. For a $D$-linear map $f:V\rightarrow V'$ of right vector spaces over $D$, the restriction $f^\natural:V^\natural \rightarrow {V'}^\natural$ of $f$ is an $F$-linear map. For a zero or an $\epsilon$-hermitian form $h$ on $V$, we define a bilinear form $h^\natural$ on $V^\natural$ such that 
\[
h(xe,ye) = \left( \begin{array}{cc} 0 & 0 \\ h^\natural(xe,ye) & 0 \end{array} \right) 
\] 
for $x,y \in V$. Then, the following properties hold:
\begin{enumerate}
 \item if $h=0$, then $h^\natural=0$;  
 \item if $h$ is hermitian, then $h^\natural$ is symplectic;
 \item if $h$ is skew-hermitian, then $h^\natural$ is symmetric.
\end{enumerate}
Put $\V^\natural= (V^\natural, h^\natural)$ and 
\[
G^\natural = \Isom(\V^\natural) 
= \{ g \in \GL(V^\natural;F) \mid h^\natural(gv, gw) = h^\natural(v,w) \mbox{ for } v,w \in V^\natural\}. 
\]
Then we have an isomorphism $\iota_G: G \rightarrow G^\natural: g \mapsto g^\natural$.
If $\pi$ is a representation of $G$, then we denote by $\pi^\natural$ the representation of $G^\natural$ such that $\pi = \pi^\natural\circ \iota_G$.


\section{
            Doubling zeta integrals
            }\label{dzi}


\subsection{
                Doubling $\epsilon$-hermitian spaces and unitary groups 
                }\label{doubledsp}
Let $\V^\Box = (V^\Box,h^\Box)$ be a pair, where $V^\Box= V\times V$ and $h^\Box = h \oplus (-h)$, that is, the map defined by
\[
h^\Box((v_1,v_2),(w_1,w_2)) = h(v_1,w_1)-h(v_2,w_2)
\]
for $v_1,v_2,w_1,w_2 \in V$. Let $G^\Box=\Isom(\V^\Box)$. Then $G\times G$ acts on $V \times V$ by 
\[
(g_1, g_2)\cdot (v_1,v_2) = (g_1v_1,g_2v_2), 
\]
so that $G \times G$ can be embedded naturally in $G^\Box$. Consider the maximal totally isotropic subspaces
\begin{align*}
V^\tru &= \{ (v,v) \in V^\Box \mid v \in V\}, \\
V^\trd & =\{ (v,-v) \in V^\Box \mid v \in V\}.
\end{align*}
Note that $V = V^\tru \oplus V^\trd$. Then $P(V^\tru)$ is a maximal parabolic subgroup of $G^\Box$ and its Levi part is isomorphic to
\[
\begin{cases}
\GL(V^\tru) \times \GL(V^\Box/V^\tru) & \mbox{ in the linear case} \\
\GL(V^\tru) & \mbox{ in the $\epsilon$-hermitian case}.
\end{cases}
\]


\subsection{
                Zeta integrals and intertwining operators 
                }\label{zeta}

Assume that $F$ is a local field of characteristic zero. Denote by $\Delta_{V^\tru}$ the character of $P(V^\tru)$ given by
\[
\Delta_{V^\tru}(x) = \begin{cases}
                            N_{V^\tru}(x) \cdot N_{V^\Box/V^\tru}(x)^{-1} 
                                         & \mbox{ in the linear case}, \\
                            N_{V^\tru}(x) & \mbox{ in the $\epsilon$-hermitian case}.                                
                            \end{cases}
\]
Here, $N_{V^\tru}(x)$ ({resp. }$N_{V^\Box/V^\tru}(x)$) is the reduced norm of the image of $x$ in $\End_DV^\tru$ (resp. $\End_DV^\Box/V^\tru$).  
Let $\omega:F^\times\rightarrow \C^\times$ be a character. For $s \in \C$, put $\omega_s = \omega \cdot | \cdot |^s$. Choose a maximal compact subgroup $K$ of $G^\Box$ such that $G^\Box = P(V^\tru)K$. Denote by $I(s,\omega)$ the degenerate principal series representation
\[
\Ind_{P(V^\tru)}^{G^\Box}(\omega_s\circ \Delta_{V^\tru})
\]
consisting of the smooth right $K$-finite functions $f:G^\Box \rightarrow \C$ satisfying
\[
 f(pg) = \delta_{P(V^\tru)}^{\frac{1}{2}}(p) \cdot \omega_s(\Delta_{V^\tru}(p)) \cdot f(g)
\]
for $p \in P(V^\tru)$ and $g \in G^\Box$, where $\delta_{P(V^\tru)}$ is the modular function of $P(V^\tru)$. We may extend $|\Delta_{V^\tru}|$ to a right $K$-invariant function on $G^\Box$ uniquely. For $f \in I(0,\omega)$, put $f_s = f \cdot |\Delta_{V^\tru}|^s \in I(s,\omega)$. Define an intertwining operator $M(s,\omega): I(s,\omega) \rightarrow I(-s,\omega^{-1})$ by
\[
(M(s,\omega)f_s)(g) = \int_{U(V^\tru)}f_s(w_1ug) \:du
\]
where $w_1 = (1,-1) \in G \times G \subset G^\Box$. This integral defining $M(s,\omega)$ converges absolutely for $\Re s \gg 0$ and admits a meromorphic continuation to $\C$. 
Let $\pi$ be an irreducible representation of $G$. For a matrix coefficient $\xi$ of $\pi$, and for $f \in I(\omega,0)$, define the zeta integral by
\[
Z^\V(f_s,\xi) = \int_{G}f_s((g,1))\xi(g) \:dg. 
\]
Then the zeta integral satisfies the following properties, which is stated in \cite[Theorem $4.1$]{Yam14}. This gives a generalization of \cite[Theorem 3]{LR05}. 

\begin{thm}\label{zeta_int}
\begin{enumerate}
\item The integral $Z^\V(f_s,\xi)$ converges absolutely for $\Re s \gg 0$
         and extends to a meromorphic function in $s$. Moreover, if $F$ is non-archimedean,
         the function $Z^\V(f_s,\xi)$ is a rational function of $q^{-s}$. 
         Here $q$ denotes the cardinality of the residue field of $F$.
\item There is a meromorphic function $\Gamma^\V(s,\pi,\omega)$ such that 
\[
Z^\V(M(s, \omega)f_s,\xi) = \Gamma^\V(s,\pi,\omega) Z^\V(f_s,\xi)
\]
for all matrix coefficient $\xi$ of $\pi$ and $f_s \in I(s,\omega)$.
\end{enumerate}
\end{thm}


\section{
            Intertwining operator and Whittaker normalization
            }\label{nio}


\subsection{
                Degenerate Whittaker functionals
                }\label{dwf}

We use the notation of \S\ref{qug} and \S\ref{dzi}. Note that $D$ is possibly split. We regard $\u(V^\tru)$ as a subspace of $\End_D(V^\Box)$ and we denote by $\u(V^\tru)_{reg}$ the set of $A \in \u(V)$ of rank $n$.
Fix a non-trivial additive character $\psi: F \rightarrow \C^\times$ and $A \in \u(V^\tru)_{reg}$. We define
\[
\psi_A: U(V^\trd) \rightarrow \C^\times: u \mapsto \psi(\operatorname{Tr}_{V^\Box}(uA))
\]
where $\operatorname{Tr}_{V^\Box}$ denotes the reduced trace of $\End_D(V^\Box)$. For $f \in I(\omega,0)$ we define
\[
l_{\psi_A}(f_s) = \int_{U(V^\trd)}f_s(u)\psi_A(u) \:du.
\]
Then this integral defining $l_{\psi_A}$ converges for $\Re s \gg 0$ and admits a holomorphic continuation to $\C$. For the proof, see \cite[\S3.3]{Kar79} in the non-archimedean case, \cite[Theorem 7.1, Theorem 7,2]{Wal88} in the archimedean case. 
The functional $l_{\psi_A}$ is called a \emph{degenerate Whittaker functional}, which is a $(U(V^\trd),\overline{\psi_A})$-equivariant functional on $I(s,\omega)$. 
On the other hand, the space of $(U(V^\trd),\overline{\psi_A})$-equivariant functionals on $I(s,\omega)$ is one dimensional for all $s \in \C$ (see \cite[Theorem 3.2]{Kar79}).
Hence we have the following proposition.

\begin{prop}
There is a meromorphic function $c(s,\omega,A,\psi)$ of $s$ such that 
\[
l_{\psi_A} \circ M(s,\omega) = c(s,\omega,A,\psi) l_{\psi_A}.
\]
\end{prop}

Then we define the normalized intertwining operator 
\[
M^*(s,\omega,A,\psi) = c(s,\omega,A,\psi)^{-1}M(s,\omega)
\]
and put
\[
\Gamma^\V(s,\pi,\omega,A,\psi) = c(s,\omega,A,\psi)^{-1} \Gamma^\V(s,\pi,\omega).
\]
Clearly, $\Gamma^\V(s,\pi,\omega,A,\psi)$ is a meromorphic function of $s$ satisfying
\begin{align}\label{Ga*}
Z(M^*(s,\omega,A,\psi)f_s,\xi) = \Gamma^\V(s,\pi,\omega,A,\psi)Z(f_s,\xi)
\end{align}
for any $f \in I(\omega,0)$ and any matrix coefficient $\xi$ of $\pi$. Note that the function $\Gamma^\V(s,\pi,\omega,A,\psi)$ does not depend on the choice of measure on $G$.


\subsection{
                The normalization factor
                }\label{pnf}

In this subsection, we study the basic properties of the normalizing factor $c(s,\omega,A,\psi)$.
First, we give an explicit formula of $c(s,\omega,A,\psi)$. Second, we study the dependence of $C(s,\omega,A,\psi)$ by the change of $\psi$.

To give an explicit formula, it is necessary to specify the Haar measure $du$ in the definition of $M(s,\omega)$ (\S\ref{zeta}). Note that we do not have to pay attention to the choice of Haar measure in the later sections since the normalized intertwining operator $M(s,\omega,A,\psi)$ does not depend on that. Let $v_1, \cdots, v_n$ be a basis of $V$, and let $e_1, \ldots, e_{2n}$ be a basis   
\[
e_j = \begin{cases}
       (v_j,v_j) & 1 \leq j \leq n \\
       (v_{j-n},-v_{j-n}) & n+1 \leq j \leq 2n.
       \end{cases}
\]
Then, we may regard $G^\Box \subset \GL_{2n}(D)$ by this basis.  
In the linear case, putting $\u=\M_n(D)$, we have a bijection 
\[
\iota: \u \rightarrow \u(V^\tru): X \mapsto \begin{pmatrix} 0 & X \\ 0 & 0 \end{pmatrix}.
\]
In the $\epsilon$-hermitian case, putting
\[
\u = \{ X \in \M_n(D) \mid {}^t\!X^* = -\epsilon X \},
\]
we have a bijection
\[
\iota: \u \rightarrow \u(V^\tru): X \mapsto \begin{pmatrix} 0 & XR \\ 0 & 0 \end{pmatrix}
\]
where $R = (h(v_i, v_j))_{ij} \in \GL_n(D)$. Let $d^\psi X$ be a self-dual Haar measure on $\u$ with respect to the pairing
\[
\u \times \u \rightarrow \C: (X,Y) \mapsto \psi(T(XY)).
\]
In this subsection, we choose the push-forward measure $\iota_*(d^\psi X)$ for $du$.

Now we state the explicit formula of $c(s,\omega,A,\psi)$. Note that we do not use it in the proof of the main theorem. 
We denote $e(G)$ the invariant of Kottwitz \cite{Kot83}. If $D$ is split, then we have $e(G)=1$. If $D$ is not split, then we have
\[
e(G) = \begin{cases}
         (-1)^n & \mbox{in the linear case} \\
         (-1)^{\frac{1}{2}n(n-1)} & \mbox{in the skew-hermitian case} \\
         (-1)^{\frac{1}{2}n(n+1)} & \mbox{in the hermitian case}.
         \end{cases}
\]
Let $A \in \u(V^\tru)_{reg}$. We define $R(s,\omega,A,\psi)$ by
\[
R(s,\omega,A,\psi) = \begin{cases}
    \omega_s(N_V(\frac{1}{2}A))^{-2} & \mbox{in the linear case},\\
     \omega_s(N_V(A))^{-1} \gamma(s+\frac{1}{2},\omega\chi_{\disc(A)},\psi)
    \epsilon(\frac{1}{2},\chi_{\disc(A)},\psi)^{-1} & \mbox{in the hermitian case},\\
     \omega_s(N_V(A))^{-1} \epsilon(\frac{1}{2},\chi_{\disc(\V)},\psi)
                                                & \mbox{in the skew-hermitian case}.
                              \end{cases}
\]
\begin{prop}\label{norm_ex}
\begin{enumerate}
\item In the linear case, we have \label{norm_ex1}
\[
c(s,\omega,A,\psi) = e(G)\cdot|2|^{-4ns}\omega^{-2n}(4)
                           \cdot \prod_{i=0}^{2n-1}\gamma(2s-i,\omega^2,\psi)^{-1}
                           \cdot R(s,\omega,A,\psi)^{-1},
\]
\item In the $\epsilon$-hermitian case, we have 
\[
c(s,\omega,A,\psi) = e(G)\cdot|2|^{-2ns+n(n-\frac{1}{2})}\omega^{-n}(4)
                            \cdot\prod_{i=0}^{n-1}\gamma(2s-2i,\omega^2,\psi)^{-1} 
                            \cdot R(s,\omega,A,\psi)^{-1}.
\]
\end{enumerate}
\end{prop}

\begin{proof}
In the linear case, by the analogue of \cite[Lemma 3.1]{Ike17}, we have
\[
c(s,\omega,A,\psi) = e(G)\omega_{2s}^2(N_V(A))\gamma_{\GL_n(D)}^{GJ}(2s-n+\frac{1}{2}, \omega^2\circ N, \psi)^{-1}
\]
where the $\gamma$-factor in the right hand side is the Godement-Jacquet $\gamma$-factor of the representation $\omega^2\circ N: \GL_n(D) \rightarrow \C^\times$. Hence we have the claim \eqref{norm_ex1}. In the skew-hermitian case, the claim is proved in \cite[Theorem 3.1]{Yam17}. In the hermitian case, the claim is proved in \cite[Appendix]{Igu86} considering the analogue of \cite[Lemma 3.1]{Ike17}.  
\end{proof}


In the later part of this subsection, we study the dependence of $c(s,\omega,A,\psi)$ by the change of $\psi$. For $a \in F^\times$, we denote by $\psi_a$ the the additive character $x \mapsto \psi(ax)$ of $F$, and we denote
\[
   T_N(s,\omega,a)
    =\begin{cases}
     \omega_{s-\frac{1}{2}}(a)^N 
        &  \mbox{ in the linear case, the hermitian case}, \\
     \omega_{s-\frac{1}{2}}(a)^N\chi_{\disc(\V)}(a) 
        &  \mbox{ in the skew-hermitian case}
     \end{cases}
\]
where 
\[
N = \begin{cases}
   4n & \mbox{ in the linear case }, \\
   2n+1 & \mbox{ in the hermitian case }, \\
   2n & \mbox{ in the skew-hermitian case }.
   \end{cases}
\]
\begin{lem}\label{whn1}
Let $a \in F^\times$. Then we have
  \[
    c(s,\omega,A,\psi_a) = T_N^{-1}(s,\omega,a) \cdot c(s,\omega,A,\psi).
   \]
\end{lem}

\begin{proof}
We can prove this lemma as in \cite[Lemma 10]{LR05}. 
Note that we can also prove it directly from Proposition \ref{norm_ex}.
\end{proof}


\section{
            Statement of the main theorem
            }\label{sm}


\subsection{
                Definition of the $\gamma$-factor
                }
                 \label{defgamma}

We use the setting of \S\ref{qug}-\S\ref{nio}. Note that $D$ is possibly split.
Fix $A \in \u(V^\tru)_{reg}$. Then the image of $A$ is  $V^\tru$. Define $\varphi_A \in \End_D(V)$ so that the following diagram is commutative:
\[
\xymatrix{
              V^\trd \ar[r]^-{A} \ar[d]_p & V^\tru \ar[d]^p \\
                    V  \ar[r]_-{\varphi_A} & V 
              }
\]
where $p$ is the first projection of $V\times V$ to $V$. We define 
\begin{align}\label{N_V}
N_V(A) := N_V(\varphi_A), \ \ \disc(A) := (-1)^n N_V(A) \in F^\times/F^{\times 2}.
\end{align}
For $\disc \in F^\times/F^{\times 2}$, denote by $\chi_\disc$ the character $F^\times \rightarrow \C^\times: x \mapsto (x,\disc)_F$ where $(\cdot ,\cdot)_F$ is the Hilbert symbol of $F$.

\begin{df}\label{dfgm}
Let $\pi$ be an irreducible representation of $G$, let $\omega$ be a character of $F^\times$, and let $\psi$ be a non-trivial character of $F$. Then we define the $\gamma$-factor of $\pi$ by
\[
\gamma^\V(s+\frac{1}{2},\pi\times\omega,\psi)
           = \Gamma^\V(s,\pi,\omega,A,\psi) \cdot c_\pi(-1) \cdot R(s,\omega,A,\psi)
\]
where $\Gamma^\V(s,\pi,\omega,A,\psi)$ is the meromorphic function defined by \eqref{Ga*},  $R(s, \omega,A,\psi)$ is a meromorphic function defined in \S\ref{pnf}, and $c_\pi$ is the central character of $\pi$.
\end{df}

By Proposition \ref{norm_ex}, the factor $\gamma^\V(s,\pi\times\omega,\psi)$ does not depend on the choice of $A$.
Note that we interpret $N_V(A) = \disc(A) = 1, \disc(\V)=1$ and $\Gamma^\V(s,\pi,\omega,A,\pi) = 1$ when $n=0$. 

We can also define the $L$-factor and the  $\epsilon$-factor as in \cite[\S10]{LR05}. Note that Yamana gave another definition of the $L$-factor by g.c.d property and showed that both $L$-factors coincide \cite{Yam14}.


\subsection{
                Lapid-Rallis $\gamma$-factors. 
                }\label{LRdf}

If $D$ is split, then $\dim_FV^\natural=2n$ and $h^\natural$ is either zero, symplectic form or symmetric form (see \S\ref{morita}). In this subsection, we consider the Lapid-Rallis $\gamma$-factor defined in \cite[\S9]{LR05}. We will prove that the Lapid-Rallis $\gamma$-factor coincides with the $\gamma$-factor defined in Definition \ref{dfgm} (see Theorem \ref{maingamma} \eqref{split} below). Note that we need to treat the case where ``$A$ is not split'' (for definition, see the end of \S5 in \cite{LR05}), since such $A$ may appear as a localization.

Let $F$ be a local field of characteristic $0$, and let $\V$ be a pair $(V,h)$ consisting of a  $2n$-dimensional vector space $V$ and a bilinear map $h:V\times V \rightarrow F$.
We assume that $h$ is either zero, a non-degenerate symplectic form, or a non-degenerate symmetric form. Let $G$ denote $\Isom(\V)$, let $\pi$ be an irreducible representation of $G$, let $\omega$ be a character of $F^\times$, and let $\psi$ be a non-trivial additive character of $F$. 
In this subsection, $\V^\Box$ denotes $(V\times V, h^\Box)$ where $h^\Box = h \oplus (-h)$, $V^\tru$ denotes a totally isotropic subspace $\{(v, v) \mid v \in V\}$ of $\V^\Box$, and $\u(V^\tru)$ denotes the Lie algebra of the unipotent radical of the parabolic subgroup of $\Isom(\V^\Box)$ corresponding $V^\tru$. We regard $\u(V^\tru)$ as a subspace of $\End_F(V^\tru)$ and we denote by $\u(V^\tru)_{reg}$ the set of $A \in \u(V^\tru)$ of rank $2n$. 
Take a basis $v_1, \ldots, v_{2n}$ of $V$. Then we define the \emph{discriminant} of $\V$ by
\[
\disc(\V) := (-1)^n \det( (h(v_i,v_j))_{ij}) \in F^\times/F^{\times 2}.
\]
We define $N_V(A)$ and $\disc(A)$ for $A \in \u(V^\tru)_{reg}$ as in \S\ref{defgamma}.
Then we define the Lapid-Rallis $\gamma$-factor by
\begin{align}\label{LRdef}
\gamma^{LR}(s+\frac{1}{2},\pi\times\omega,\psi) 
 = \Gamma^\V(s,\pi,\omega,A,\psi) \cdot c_{\pi}(-1) \cdot Q(s,\omega,A,\psi)
\end{align}
where
$\Gamma^\V(s,\pi,\omega,A,\psi)$ is the $\Gamma$-factor as defined in \cite[\S5]{LR05}, 
$c_\pi$ is the central character of $\pi$, and 
\[
Q(s,\omega,A,\psi) = \begin{cases}
    \omega_s(N_V(\frac{1}{2}A))^{-2} & \mbox{ in the linear case}, \\
    \omega_s(N_V(A))^{-1} \gamma(s+\frac{1}{2},\omega\chi_{\disc(A)},\psi)
      \epsilon(\frac{1}{2},\chi_{\disc(A)},\psi)^{-1} & \mbox{ in the symplectic case}, \\
    \omega_s(N_V(A))^{-1} \epsilon(\frac{1}{2}, \chi_{\disc(\V)},\psi) 
                                             & \mbox{ in the symmetric case}.
     \end{cases}
\]
By Proposition \ref{norm_ex}, The right hand side of \eqref{LRdef} does not depend on $A \in \u(V^\tru)_{reg}$.
Lapid and Rallis moreover defined the $L$-factor and the $\epsilon$-factor in \cite[\S10]{LR05}. We denote them by $L^{LR}(s,\pi\times\omega)$ and $\epsilon^{LR}(s,\pi\times\omega,\psi)$.


\subsection{
            Some Remarks
           }\label{errors}

The definition \cite[(25)]{LR05} seems not to be correct, but this can be fixed as follows.

\begin{rem}
We note here the corrections in the linear case.
First, the factor $\pi(-1)$ of \cite[(25)]{LR05} should be replaced with $\pi \otimes\omega(-1)$
(see \cite[\S7.2]{Gan12}).
Second, the factor $\theta(\det_VA))^{-1}$ of \cite[(25)]{LR05} should be replaced with 
$\omega_s(\theta(\det_V(A))^{-1}$.
Consequently, the Lapid-Rallis $\gamma$-factor should be defined by
\[
\gamma^\V(s+\frac{1}{2},\pi\times\omega,\psi) = \Gamma^\V(s,\pi,\omega,A,\psi)
   \cdot c_{\pi\otimes\omega}(-1)\cdot Q(s,\omega,A)
\]
where $\Gamma^\V(s,\pi,\omega,A,\psi)$ is the $\Gamma$-factor as defined in \cite[\S5]{LR05}, $c_{\pi\otimes\omega}$ is the central character of $\pi\otimes\omega$ and
\[
Q(s,\omega,A,\psi) = \omega_s({\det}_V(A))^{-1}\omega_s(\theta({\det}_V(A)))^{-1}.
\]
Note that $\omega$ is a character of $E^\times$ where $E=F$ of $E$ is a quadratic extension of $F$.
\end{rem}

\begin{rem}\label{thirderror}
In the symmetric case, the symplectic case, and the hermitian case, the factor $\omega_s(\det_V(A))$ in \cite[(25)]{LR05} should be replaced with $\omega_s(\det_V(2A))$.
See Remark \ref{Lem8} for the necessity of this modification. Since
\[
{\det}_V(2A) = N_V(A)
\]
where $\det_V(\cdot)$ and $N_V(\cdot)$ are defined in \cite[p.337]{LR05} and the analogue of \eqref{N_V} respectively, 
our definition of the Lapid-Rallis $\gamma$-factor \eqref{LRdef} is consistent with this modification.
\end{rem}

\begin{rem}
In the case of hermitian spaces over a quadratic extension $E$ of $F$, the definition \cite[(25)]{LR05} needs to be modified more: let
\[
\epsilon(\V) = \chi((-1)^{n(n-1)/2} \det((h(v_i,v_j))_{ij}))
\] 
where $\chi$ is the non-trivial quadratic character of $F^\times/N_{E/F}(E^\times)$, 
and $(v_1,\ldots,v_n)$ is a basis of $V$ over $E$. Then as explained in \cite[\S10.1]{GI14}, the right hand side of \cite[(25)]{LR05} should be multiplied 
by $\epsilon(\V)^{n+1}$. Note that $\epsilon(\V)^{n+1} = 1$ when $n$ is odd.
Consequently, in the hermitian case over $E$, the Lapid-Rallis $\gamma$-factor should be defined by
\[
\gamma^\V(s+\frac{1}{2},\pi\times\omega,\psi) = \Gamma^\V(s,\pi,\omega,A,\psi)
   \cdot c_{\pi\otimes\omega}(-1)\cdot Q(s,\omega,A)
\] 
where $\Gamma^\V(s,\pi,\omega,A,\psi)$ is the $\Gamma$-factor as defined in \cite[\S5]{LR05}, $c_{\pi\otimes\omega}$ is 
the central character of $\pi\otimes\omega$ and 
\[
Q(s,\omega,A) = \epsilon(\V)^{n+1}{\det}_V(2A).
\] 
\end{rem}

\begin{rem}
Recall that in the hermitian case, the character $\psi_A$ of \cite[\S5]{LR05} is defined by 
$\psi_A(X) = \psi_F(\operatorname{tr}_{E/F}(\operatorname{tr}(XA)))$, but Gan-Ichino \cite[p539]{GI14} said that 
$\psi_A$ should be given by $\psi_A(X) = \psi_F(\operatorname{tr}(XA))$. Taking this into account, 
the definition of $\gamma$-factor stated in \cite[\S10.1]{GI14} is correct as stated in this case. 
However, in the other cases (Case B,C',C'',D), their definition \cite[\S10.1]{GI14} requires the same modification explained above.
\end{rem}

\begin{rem}
As with the Lapid-Rallis $\gamma$-factor, the $\gamma$-factor for metaplectic groups defined in \cite[\S5]{Gan12} needs to be modified. 
The factor $\det(A)$ appeared in the expression (which defines the $\gamma$-factor) in the fourth line from the bottom of p.76 should be replaced with $\det(2A)$. 
\end{rem}


\subsection{
                Main theorem
                }
                \label{mainthm}
We use the setting of \S\ref{defgamma}.
For a non-trivial character $\psi$ of $F$ and an irreducible representation $\rho$ of $\GL_m(D)$, we can attach the ``Godement-Jacquet $\gamma$-factor'' as
\[
\gamma^{GJ}(s,\rho,\psi)
 = \varepsilon^{GJ}(s,\rho,\psi) \frac{L^{GJ}(1-s,\tilde{\rho})}{L^{GJ}(s,\rho)}
\]
where $\tilde{\rho}$ is the contragredient representation of $\rho$ and $L^{GJ}(s,-)$ (resp. $\varepsilon^{GJ}(s,-,\psi)$) is the Godement-Jacquet $L$-factor (resp. $\epsilon$-factor) (see \cite[Theorems 3.3, 8.7]{GJ72}).

\begin{thm}[Main]\label{maingamma}
The factor $\gamma^\V(s,\pi\times\omega,\psi)$ satisfies the following properties:
\begin{enumerate}
 \item (unramified twisting)
         \[
         \gamma^\V(s,\pi\times\omega_{s_0},\psi) = \gamma^\V(s+s_0,\pi\times\omega,\psi)
         \]\label{ut}
          for $s_0 \in \C$.
 \item (multiplicativity)
         Let $W$ be a totally isotropic subspace of $V$, and let $\sigma = \sigma_0 \otimes\sigma_1$ be an irreducible representation of $\GL(W) \times \Isom(\W_1)$ (see \S\ref{parabsub}). 
         If $\pi$ is a constituent of $\Ind_{P(W)}^G(\sigma)$, then
         \[
         \gamma^\V(s,\pi\times\omega,\psi) = \gamma^{\W_0}(s,\sigma_0\times\omega,\psi)
                                                             \gamma^{\W_1}(s,\sigma_1\times\omega,\psi).
         \] \label{ml}
 \item (split factor)
         If $D$ is split, then
         \[
         \gamma^\V(s,\pi\times\omega,\psi) 
                     = \gamma^{LR}(s,\pi^\natural\times\omega,\psi).
         \]
         \label{split}
 \item (functional equation)
         \[
         \gamma^\V(s,\pi\times\omega,\psi)
          \gamma^\V(1-s,\tilde{\pi}\times\omega^{-1},\psi^{-1}) = 1.
         \]\label{fe}
 \item (self duality)
         \[
         \gamma^\V(s,\tilde{\pi}\times\omega,\psi) = \gamma^\V(s,\pi\times\omega,\psi).
         \]\label{sd}
 \item (dependence on $\psi$)
         Denote by $\psi_a$ the additive character $x \mapsto \psi(ax)$ of $F$
         for $a \in F^\times$. Then
         \[
         \gamma^\V(s,\pi\times\omega,\psi_a)
               = T_N(s,\omega,a) \cdot \gamma^\V(s,\pi\times\omega,\psi)
         \]
        where $N$ and $T_N(s,\omega,a)$ are defined in \S\ref{pnf} \label{psi}
 \item (minimal cases)
          Suppose $F = \R$, $\V$ is $\epsilon$-hermitian, $n \leq 1, \omega=1$ and $\pi$ is trivial. 
          Let $\phi_\pi$ be the L-parameter of $\pi$. Then
          \[
          \gamma^\V(s,\pi\times\omega,\psi) = \gamma(s,\phi_\pi\otimes\omega,\std,\psi).
          \]
          Here, if $n=0$ we interpret the right hand side as $\gamma(s,\omega,\psi)$ (resp. $1$) when $\V$ is hermitian (resp. skew-hermitian). \label{min}
 \item ($\GL_n$-factor) 
          In the linear case, 
          \[
          \gamma^\V(s,\pi\times\omega,\psi) 
            = \gamma^{GJ}(s,\pi\otimes\omega,\psi)\gamma^{GJ}(s,\tilde{\pi}\otimes\omega,\psi).
          \] \label{GLn}
 \item (archimedean Langlands compatibility)
          If $F$ is archimedean then
          \[
          \gamma^\V(s,\pi\times\omega,\psi) = \gamma(s, \phi_\pi\otimes\omega,\std,\psi)
          \]
          where $\phi_\pi$ is the Langlands parameter corresponding to $\pi$
          and $\std$ is the standard representation of $^LG$ to $\GL_N(F)$. \label{arc}
 \item (global functional equation)
          Let $\F$ be a number field, and let $\D$ be a quaternion algebra over $\F$.
          Let $G = \Isom(\V)$ where $\V = (V,h)$ is a pair consisting of an $n$-dimensional 
          right vector space $V$ over $\D$ and either $0$, hermitian form 
          or skew-hermitian form $h$ on $V$. Let $\pi$ be an irreducible 
          cuspidal automorphic representation of $G(\A_\F)$. Then for any finite set $S$ of            
          places of $\F$ containing all places where $\D$ is not split, the functional equation
          \[ 
          L_S^{LR}(s,\pi\times\omega) 
                = \prod_{v \in S} \gamma_v^\V(s,\pi_v\times\omega_v,\psi_v)
               \cdot \epsilon_S^{LR}(s,\pi\times\omega,\psi) 
                        L_S^{LR}(1-s,\tilde{\pi}\times\omega^{-1})
          \]
          holds, where
          \[
            L_S^{LR}(s,\pi\times\omega) 
               = \prod_{v \not\in S} L^{LR}(s,\pi_v^\natural\times\omega_v)
          \]
          and
          \[
            \epsilon_S^{LR}(s,\pi\times\omega,\psi) 
               = \prod_{v \not \in S} \epsilon^{LR}(s,\pi_v^\natural\times\omega_v,\psi_v).
          \]
          \label{gfe}
\end{enumerate}
Moreover,  the properties \eqref{ut},\eqref{ml},\eqref{split},\eqref{psi},\eqref{min},\eqref{GLn} and \eqref{gfe} determine $\gamma^\V(s,\pi\times\omega,\psi)$ uniquely.
\end{thm}

\begin{rem}
The $\gamma$-factor appears in a functional equation of the zeta integrals. Suppose that $\V$ is an $\epsilon$-hermitian space for simplicity. Let $f_s \in I(s,\omega)$, and let $\xi$ be a matrix coefficient of $\pi$. We can rewrite \eqref{Ga*} as
\[
Z(M(s,\omega)f_s,\xi) = c(s,\omega,A,\psi)\Gamma^\V(s,\omega,A,\psi)Z(f_s,\xi).
\]
We choose the Haar measure $du$ in the definition of $M(s,\omega)$ as in \S\ref{pnf}. Then, by Proposition \ref{norm_ex}, we can obtain the functional equation
\begin{align}\label{locFE_zeta}
\begin{aligned}
Z^\V(M(s,\omega)f_s, \xi) = 
& e(G) \omega_\pi(-1)\gamma^\V(s+\frac{1}{2}, \pi\times\omega,\psi)  \\
& \times |2|^{-2ns+n(n-\frac{1}{2})}\omega^{-n}(4) 
   \prod_{i=0}^{n-1}\gamma(2s-2i, \omega^2,\psi)^{-1} 
    Z^\V(f_s,\xi)
\end{aligned}
\end{align}
where $e(G)$ is the invariant of Kottwitz (see. \S\ref{pnf}).
\end{rem}


\section{
            Proof of the main theorem
            }\label{prf}

In \S6 and \S7, we prove Theorem \ref{maingamma}. Once \eqref{min} minimal cases is proved, then the other parts of the proof are not difficult or are similar to \cite[Theorem 4]{LR05}.
However, the proof of the uniqueness is important: it explains the reason why the minimal cases contains only the cases where $F=\R, n=0,1$, and $\pi, \omega$ are trivial. 
In this section, we write down the proof of the uniqueness, \eqref{split} split factor, and \eqref{arc} archimedean Langlands properties for the readers. 
The minimal cases will be proved in \S7.


\subsection{
            Uniqueness
            }\label{uniqprf}

In this subsection, we prove the uniqueness of the $\gamma$-factor, which is stated at the end of Theorem \ref{maingamma}.  Let $\gamma'$ be another function satisfying the conditions \eqref{ut},\eqref{ml},\eqref{split},\eqref{psi},\eqref{min},\eqref{GLn} and \eqref{gfe} of Theorem \ref{maingamma}. Then we will prove the equation
\begin{align}\label{uniqgoal}
\gamma'(s,\pi\times\omega,\psi) = \gamma^\V(s,\pi\times\omega,\psi).
\end{align}

We denote the skew-field of Hamilton's quaternions by
\[
\H = \R\oplus \R i \oplus \R j \oplus \R k
\]
with the elements $i,j,k$ satisfying 
\[
i^2 = j^2 = k^2 = -1,\, k=ij,\, ij=-ji.
\]
We use the setting of \S\ref{defgamma}. Recall that $F$ is an arbitrary local field of characteristic zero, $D$ is a quaternion algebra over $F$, and $\V$ is an $n$-dimensional $\epsilon$-hermitian space over $D$. By the condition \eqref{ml} and \eqref{GLn}, we may assume that all the coefficients of $\pi$ have compact support. By the condition \eqref{split}, we may assume that $D$ is a division quaternion algebra. Then, we have the following:
\begin{lem}\label{garg1}
Let $\V'$ be an $\epsilon$-hermitian space over $\H$ such that $\dim \V' = n$.
Then there is a quintuple $(\F,\D,\underline{\V},\underline{\omega},\underline{\psi})$ where
\begin{itemize}
\item $\F$ is a number field such that there are two (different) places $v_1, v_2$ with $\F_{v_1} = F, \F_{v_2} = \R$,
\item $\D$ is a division quaternion algebra over $\F$ such that $\D$ is not split precisely at the two places $v_1,v_2$ and $\D_{v_1} = D, \D_{v_2} = \H$,
\item $\underline{\V}$ is an $\epsilon$-hermitian space over $\D$ such that $\underline{\V}_{v_1} = \V, \underline{\V}_{v_2} = \V'$
\item $\underline{\omega}$ is a Hecke character of $\F$ such that $\omega \cdot\underline{\omega}_{v_1}^{-1} = | \cdot |_{v_1}^t$ for some $t \in \C$ and $\underline{\omega}_{v_2} = 1$,
\item $\underline{\psi}$ is a non-trivial additive character of $\A/\F$ where $\A$ is the ring of adeles of $\F$.
\end{itemize}
\end{lem}
\begin{proof}
The existence of such $\F, \D$, and $\omega$ is well-known. Besides, by using the weak approximation, we have an $\epsilon$-hermitian space $\underline{\V}$ satisfying the condition of the lemma.
\end{proof}
We apply this lemma to $\V'$ having a $\lfloor n/2 \rfloor$-dimensional totally isotropic subspace. By \cite[Appendice I]{Hen87}, there is an irreducible automorphic cuspidal representation $\Pi$ of $\Isom(\underline{\V})(\A)$ such that $\Pi_{v_1} = \pi$. Then, by the conditions \eqref{ut}, \eqref{split}, \eqref{psi}, and \eqref{gfe}, 
we have
\[
\gamma_{v_1}'(s,\Pi_{v_1}\times\underline{\omega}_{v_1},\underline{\psi}_{v_1})
\gamma_{v_2}'(s,\Pi_{v_2}\times\underline{\omega}_{v_2},\underline{\psi}_{v_2})
=\gamma_{v_1}^\V(s,\Pi_{v_1}\times\underline{\omega}_{v_1},\underline{\psi}_{v_1})
\gamma_{v_2}^\V(s,\Pi_{v_2}\times\underline{\omega}_{v_2},\underline{\psi}_{v_2}).
\]
Thus, we can reduce the equation \eqref{uniqgoal} to the case $F = \R$, $\omega=1$ and $\V$ has a $\lfloor n/2\rfloor$-dimensional totally isotropic subspace. By Casselman's embedding theorem (\cite[Corollary 5.2]{CO78}) and  the condition \eqref{ml}, we may assume that $n=0,1$. 

Again by the global argument as above, we may assume moreover that $\pi$ is trivial. In this case, the equation \eqref{uniqgoal} clearly holds by the condition \eqref{min}. 


\subsection{
           Formal properties
           }\label{fp}

The properties \eqref{ut}, \eqref{ml}, \eqref{fe}, \eqref{sd}, \eqref{psi}, and \eqref{gfe} can be deduced from the framework of the doubling method. They can be proved as in \cite[\S9]{LR05}.
However we explain \eqref{ml} here to give the detail of Remark \ref{thirderror}.

We use the setting of \eqref{ml}.
We may assume that $\pi$ is a constituent of $\Ind_{P(W)}^{G^\Box}\sigma$. We denote by $\W$ the orthogonal sum $\W_0 \bot \W_1$ (\S\ref{parabsub}).
We define $I^\W(s,\omega)$ as in \S\ref{zeta} and we define an intertwining map
\[
\Psi(s,\omega): I^\V(s,\omega) \rightarrow \Ind_{P(W^\Box)}^{G^\Box}(I^\W(s,\omega)\otimes|\Delta_{\W:\V}|): f_s \mapsto (g \mapsto [\Psi(s,\omega)f_s]_g)
\]
as in \cite[Proposition 1]{LR05}. Fix $A \in \u(V^\tru)_{reg}$. We may assume $A(W^\Box) \subset W^\Box$. Then $A$ induces the following maps
\begin{itemize}
\item $A_0: W^\Box \rightarrow W^\Box$,
\item $A_1:(W^\bot/W)^\Box \rightarrow (W^\bot/W)^\Box$.
\end{itemize}
For $f_s' \in I^\W(s,\omega)$ and for
\[
(B,C) \in \u((W\bot (W^\bot/W))^\tru) = \u(W^\tru)\times \u((W^\bot/W)^\tru),
\]
we define $l_{\psi_{(B,C)}}^\W(f_s')$ as in \S\ref{dwf}. Then, we have the following:
\begin{prop}\label{lm}
Let $f_s \in I(s,\omega)$.
\begin{enumerate}
\item In the linear case, we have 
        \[
         l_{\psi_A}(f_s) = \int_{U(V^\trd)\cap P(W^\Box)\backslash U(V^\trd)}
                            l_{\psi_{(A_0,A_1)}}^\W([\Psi(s)f_s]_u)\psi_A(u) \:du.
        \]
\item In the $\epsilon$-hermitian case, we have  
        \[
          l_{\psi_A}(f_s) = \int_{U(V^\trd)\cap P(W^\Box)\backslash U(V^\trd)}
                            l_{\psi_{(2A_0,A_1)}}^\W([\Psi(s)f_s]_u)\psi_A(u) \:du.
        \]
\end{enumerate}
\end{prop}

\begin{rem}\label{Lem8}
Proposition \ref{lm} corrects the inaccuracy in the statement of \cite[Lemma 8]{LR05}. The functional $l_{\psi_B}^\W$ in \cite[Lemma 8]{LR05} corresponds to $l_{\psi_{(A_0,A_1)}}^\W$ in our notation. However, it should be replaced with $l_{\psi_{(2A_0, A_1)}}^\W$ in the symplectic case, the symmetric case, and the hermitian case. This causes the modification explained in Remark \ref{thirderror} of \S \ref{errors}.
\end{rem}

As in \cite[\S9]{LR05}, we have
\[
\Gamma^\V(s,\pi,\omega,A,\psi) = \Gamma^{\W_0}(s,\sigma_0,\omega,2A_0,\psi)
                                               \cdot \Gamma^{\W_1}(s,\sigma_1,\omega,A_1,\psi).
\]
Observe that 
\[
R^{\V}(s,\omega,A,\psi) = R^{\W_0}(s,\omega,2A_0,\psi) \cdot R^{\W_1}(s,\omega,A_1,\psi),
\]
thus we have
\[
\gamma^\V(s,\pi\times\omega,\psi) = \gamma^{\W_0}(s,\sigma_0\times\omega,\psi)
                                \cdot \gamma^{\W_1}(s,\sigma_1\times\omega,\psi),
\]
as desired.


\subsection{
           Split case
           }\label{sc}

In this subsection, we prove the property \eqref{split}. Suppose that $D$ is split. Take $A \in \u(V^\tru)_{reg}$. Then one can express $I(s,\omega), Z(-,\xi)$ and $l_{\psi_A}$ in the term of $V^\natural, \pi^\natural$ and $A^\natural$ where $A^\natural$ corresponds to $A$ via the Morita equivalence between $\End_F(V^{\Box\natural})$ and $\End_{D}(V^\Box)$ (\S\ref{morita}). Note that $V^{\Box\natural} = V^{\natural\Box}$ and $A^\natural\in\u(V^{\natural\tru})_{reg}$ (\S\ref{LRdf}).
Then we have 
\begin{align*}
\Gamma^{\V^\natural}(s, \omega, \pi^\natural, A^\natural, \psi)
 &= \Gamma^\V(s, \omega, \pi, A, \psi), \\
Q^{\V^\natural}(s,\omega,A^\natural,\psi) & = R^\V(s,\omega,A,\psi).
\end{align*}
Thus, we have the property \eqref{split}.


\subsection{
            Relation with the local Langlands Correspondence.
            }\label{relLLC}

The properties \eqref{min}, \eqref{GLn}, and \eqref{arc} of Theorem \ref{maingamma} say that the equation \eqref{desi} is true as desired in special cases. 
Note that the property \eqref{GLn} is proved by Yamana \cite[Appendix]{Yam13}. 
The property \eqref{min} will be proved in \S7. In this subsection, we prove the property \eqref{arc} admitting the consequence of \S7. 
Moreover we consider the non-archimedean case (\S\ref{nonarc}). 

\subsubsection{
                Notations
                }\label{notations}

Before starting the proof, we give some notations involving the quaternionic unitary groups over archimedean local fields and their representations.
We use the setting of \S\ref{defgamma}.
Suppose that $F$ is archimedean. If $D$ is split over $F$, the equation \eqref{desi} is proved by Lapid and Rallis \cite{LR05}. Hence, we may assume $F=\R$ and $D=\H$. By unramifed twisting, we may assume that $\omega = 1$ or the sign character $\sgn$. Moreover, by the multiplicativity, we may assume that $\V = (\H^n, \langle I_n \rangle)$ (resp. $\V = (\H, \langle i \rangle)$) in the hermitian case (resp. the skew-hermitian case).
Here, 
\begin{align*}
&\langle I_n \rangle: \H^n \times \H^n \rightarrow \H: (x,y) \mapsto {}^t\!x^*\cdot y, \\
&\langle i \rangle: \H\times \H \rightarrow \H: (x,y) \mapsto x^*iy.
\end{align*}
Then $G$ is a subgroup of $\GL_n(\H)$.  
We choose a basis $e_1, \ldots, e_{2n}$ of $(\H^n)^\Box = \H^n \times \H^n$ defined by
\[
e_j = \begin{cases}
 \frac{1}{\sqrt{2}} (v_j,v_j) & 1 \leq j \leq n, \\ 
 \frac{1}{\sqrt{2}} (v_{j-n}, -v_{j-n}) & n+1 \leq j \leq 2n,
\end{cases}
\]
where $v_1 = {}^t\!(1 \ 0 \ \cdots \ 0 ), \ldots, v_n = {}^t\!(0 \ \cdots \ 0 \ 1)$. We may regard $G^\Box \subset \GL_{2n}(\H)$ by this basis. Put
\[
K = \{ g \in G^\Box \mid g\cdot {}^t\!g^* = 1 \}.
\]
Then $K$ is a maximal compact subgroup of $G^\Box$, and the embedding 
\[
G\times G \rightarrow K : (a,b) \mapsto w_0 \begin{pmatrix} a & 0 \\ 0 & b \end{pmatrix} w_0^{-1}
\]
is an isomorphism. Here
\[
w_0 := \frac{1}{\sqrt{2}}\begin{pmatrix} I_n & I_n \\ I_n & -I_n \end{pmatrix}.
\]
Moreover, we have the decomposition $G^\Box = P(V^\tru)K$ by the analogue of \cite[Lemma 2.1]{GPSR87}.

In the later part of this subsection, we explain the finite dimensional representations of the Weil group $W_\R$ of $\R$. 
We regard $\C$ as a subfield of $\H$ by identifying $i\in \H$ with the imaginary unit of $\C$.
The Weil group $W_\R$ of $\R$ is given by $\C^\times \cup j \C^\times \subset \H^\times$. For any character $\omega'$ of $\R^\times$, we also denote by $\omega'$ the one-dimensional representation of $W_\R$ defined by the composition $\omega' \circ \alpha$ where $\alpha: W_\R \rightarrow \R^\times$ is the homomorphism defined by 
\[
\alpha(j)= -1, \ \ \alpha(z) = z \overline{z} \ \mbox{ for } z \in \C^\times.
\]
For $l \in \Z$, we denote by $D_l$ the two-dimensional representation of $W_{\R}$ defined by
\[
D_l(j) = \begin{pmatrix} 0 & (-1)^l \\ 1 & 0 \end{pmatrix}, \,
D_l( re^{i\theta}) = \begin{pmatrix} e^{il\theta} & 0 \\ 0 & e^{-il\theta} \end{pmatrix} \mbox{ for $r \in \R_{>0}, \theta \in \R$}. 
\]
Then $D_l \cong D_{-l}$ and $D_l\otimes\sgn \cong D_l$ for $l \in \Z$. Note that all finite dimensional representations of $W_\R$ are completely reducible, and the finite dimensional irreducible representations of $W_\R$ are $1$, $\sgn$, $D_l$ for $l = 1,2, \ldots$, and their unramified twistings (cf. \cite[\S3]{Kna94}). Take the non-trivial additive character $\psi$ given by $\psi(x) = e^{2\pi i x}$ for $x \in \R$.
Put
\[
\Gamma_\R(s) := \pi^{-s/2}\Gamma(s/2), \ \ \Gamma_\C(s) := 2(2\pi)^{-s}\Gamma(s).
\]
Then, the $\gamma$-factors of $1, \sgn$ and $D_l$ for $l \in \Z$ are given by
\[
\gamma(s+\frac{1}{2},1,\std,\psi) 
= \frac{\Gamma_\R(-s+\frac{1}{2})}{\Gamma_\R(s+\frac{1}{2})}, \, 
\gamma(s+\frac{1}{2},\sgn,\std,\psi) 
= i \cdot \frac{\Gamma_\R(-s+\frac{3}{2})}{\Gamma_\R(s+\frac{3}{2})}
\] 
and
\[
\gamma(s+\frac{1}{2},D_l, \std, \psi) 
= i^{|l|+1} \cdot \frac{\Gamma_\C(-s+\frac{|l| + 1}{2})}{\Gamma_\C(s+\frac{|l|+1}{2})}.
\]

In what follows, we first describe the $\gamma$-factor of the representation $\std\circ\phi_\pi$ of $W_\R$. 


\subsubsection{
                Hermitian case
                }

We use the setting of \S\ref{notations}. Suppose that $\V$ is hermitian. Recall that $\omega=1$ or $\sgn$. In this case, $G = \Sp(n)$. Choose the maximal torus
\[
T = \{ \diag(z_1,\ldots,z_n) \in \Sp(n) \mid z_1,\ldots,z_n \in \C^\times, |z_1| = \cdots = |z_n|=1\}
\]
of $\Sp(n)$. We identify the character group $X^*(T)$ of $T$ with $\Z^n$. We fix the ``standard'' positive system. We denote by $\rho = (\rho_1, \ldots,\rho_n)$ the half sum of the positive roots of $G$. Then $\rho _j =n + 1 - j$. Let 
\[
\lambda = (\lambda_1,\ldots,\lambda_n) \in \Z^n, \ \ \lambda_1 \geq \cdots \geq \lambda_n \geq 0
\]
be the highest weight of $\pi$. Then 
\begin{align*}
 \std \circ (\phi_\pi\otimes\omega) 
           &= ((\bigoplus_{j=1}^nD_{2(\lambda_j + \rho_j)}) \oplus \sgn^n) \otimes \omega \\
           &= (\bigoplus_{j=1}^nD_{2(\lambda_j+\rho_j)}) \oplus (\sgn^n\otimes\omega))
\end{align*}
where $\sgn^n$ is the character of $W_\R$ defined by $\sgn^n(j)= (-1)^n, \sgn^n(z) = 1$ for $z \in \C^\times$.
Since
\[
\prod_{j=1}^n i^{2(\lambda_j+\rho_j)+1} = i^{n(n+2)} \prod_{j=1}^n (-1)^{\lambda_j} 
=\, \epsilon(\frac{1}{2},\sgn^n,\psi)^{-1}c_\pi(-1),
\]
we have
\begin{align*}
       & \gamma(s+\frac{1}{2},\phi_\pi\otimes\omega,\std,\psi) \\ 
\nonumber &=c_\pi(-1)\gamma(s+\frac{1}{2},\sgn^n\otimes\omega,\psi)
                    \epsilon(\frac{1}{2},\sgn^n,\psi)^{-1}
                     \prod_{j=1}^n \frac{\Gamma_\C(-s+\frac{1}{2}+\lambda_j + \rho_j)}
                                              {\Gamma_\C(s+\frac{1}{2}+\lambda_j + \rho_j)}.
\end{align*}

We next compute $\gamma^\V(s,\pi\times\omega,\psi)$. 
Observe that 
\[I(s,\omega) = I(s,1), \ \ \Gamma^\V(s,\pi,\omega,A,\psi) = \Gamma^\V(s, \pi, 1, A, \psi), \ \  \omega_s(A) = |N_V(A)|^s
\]
for $A \in \u(V^\tru)_{reg}$. Hence, if we prove the equation
\begin{align}\label{herher}
\Gamma^\V(s,\pi,1,A,\psi) = |N_V(A)|^s \prod_{j=1}^n \frac{\Gamma_\C(-s+\frac{1}{2}+\lambda_j + \rho_j)}
                                              {\Gamma_\C(s+\frac{1}{2}+\lambda_j + \rho_j)},
\end{align}
then we can conclude \eqref{desi}. We prove \eqref{herher} by induction on
\[
|\lambda| := \sum_{i=1}^v|\lambda_i|
\]
by using the ``strong adjacency'' (see \cite[p. 347]{LR05}).

We first explain the ``strong adjacency''.
Since $G^\Box = P(V^\tru)K$, letting $\hat{G}$ the set of the equivalence classes of the irreducible representations of $G$, we have the decomposition
\[
 I(0,1) = \bigoplus_{\pi\in \hat{G}} \tilde{\pi} \otimes \pi
\] 
as $G \times G$-modules. We define $i_\pi^{(s)}$ by the composition
\[
\tilde{\pi}\otimes\pi \hookrightarrow I(0,1) \rightarrow I(s,1)
\]
where $I(0,1) \rightarrow I(s,1)$ is the map $f \mapsto f_s$ (see \S\ref{zeta}).
Let $\s$ be the orthogonal complement of $\k$ in $\g^\Box$ with respect to the Killing form $\kappa$. Then $\s\otimes_\R\C$ is isomorphic to $St \otimes \tilde{St}$ as a representation of $K\cong G\times G$, where $St$ is the representation of $G$ defined by the action
\[
 g \cdot x = (g\otimes1)^\natural x
\]
for $g \in \Sp(n), x \in (\H^n \otimes_\R\C)^\natural$ (for the definition of $\natural$, see \S\ref{morita}). We identify $\s\otimes_\R\C$ with a subrepresentation of $I(0,1)$ through the map $\omega^{(\nu)}$ which is defined in \cite[1.f]{BOO96}.

\begin{df}
Let $\pi,\pi'$ be irreducible representations of $G$. We say that $\pi$ and $\pi'$ are \emph{strongly adjacent} if the image of $(\s\otimes_\R\C)\otimes(\pi\otimes\tilde{\pi})$ of the multiplication map
\[
 I(0,1) \otimes I(s,1) \rightarrow I(s,1) : f' \otimes f_s  \mapsto f'f_s
\]
contains $\pi'\otimes\tilde{\pi}'$.
\end{df}

If $\pi$ and $\pi'$ are strongly adjacent, then
\begin{align}\label{eo}
\mbox{either } \Hom_G(St\otimes \pi,\pi') \mbox{ or } \Hom_G(\pi, St\otimes\pi') \mbox{ is non-zero }.
\end{align}
Conversely, one can show that the 
 \eqref{eo} implies the strongly adjacency. By the branching rule of $St\otimes \pi$ (see the end of \cite[\S2.5]{KT87}), we have the following lemma.

\begin{lem}
Let $\pi,\pi'$ be irreducible representations of $G$. 
We denote by $\lambda = (\lambda_1,\ldots,\lambda_n)$ (resp. $\lambda' = (\lambda_1',\ldots,\lambda_n'))$ the highest weight of $\pi$ (resp. $\pi'$).
Then the following are equivalent:
\begin{enumerate}
\item $\pi$ and $\pi'$ are strongly adjacent,
\item there exists a unique integer $l $ with $1 \leq l \leq n$ such that
\[
|\lambda_j - \lambda_j'| = 
\begin{cases}
   1       & j=l \\
   0      & j\not=l
\end{cases}
\]
for $j = 1, \ldots, n$.
\end{enumerate}
\end{lem}

Now we prove the property \eqref{arc} admitting \eqref{herher} for the trivial representation.
\begin{prop}\label{ind_step}
Let $\pi,\pi'$ be irreducible representations of $G$ of highest weights $\lambda, \lambda' \in \Z^n$, respectively. 
Suppose $\pi$ and $\pi'$ are strongly adjacent, and moreover $\lambda_l' =\lambda_l+1$ for some unique $l$. Then
\[
(\lambda_l + \rho_l +\frac{1}{2} -s) \Gamma^\V(s,\pi,1,A,\psi)
  =  (\lambda_l + \rho_l + \frac{1}{2} + s) \Gamma^\V(s,\pi',1,A,\psi).
\]
\end{prop}
To prove this, we first note the following lemma:
\begin{lem}\label{eigen=gamma}
Let $\pi \in \hat{G}$, and let $A \in \u(V^\tru)_{reg}$. For $s \in \C$, the following diagram is commutative:
\[
\xymatrix{
             I(s,1) \ar[rr]^-{M^*(s,1,A,\psi)} & & I(-s,1) \\
             \tilde{\pi}\otimes\pi \ar[u]^{i_\pi^{(s)}} \ar[rr]_-{\Gamma^\V(s,\pi,1,A,\psi)\cdot}
             & & \tilde{\pi}\otimes\pi \ar[u]_{i_\pi^{(-s)}}
             }.
\]
\end{lem}
\begin{proof}
The normalized intertwining operator $M^*(s,1,A,\psi)$ acts on $\tilde{\pi}\times\pi$ as the multiplication by a scalar for almost all $s \in \C$. We denote the scalar by $b$. 
On the other hand, we have $Z(f_{-s},\xi) = Z(f_s, \xi) \not=0$ for $f \in \tilde{\pi}\otimes\pi \subset I(0,1)$ with $f\not=0$. Hence we have $b = \Gamma(s,\pi,\omega,A\psi)$, and then we have the lemma.
\end{proof}
Then, applying \cite[(2.14)]{BOO96} as in \cite[\S9]{LR05}, we have Proposition \ref{ind_step}.

Hence, to complete the induction, it only remains to verify \eqref{herher} in the case $\lambda=0$, that is, $\pi$ is trivial. 
We prove \eqref{herher} for the trivial representation in \S\ref{calcu} below to finish the proof of the property \eqref{arc}. 


\subsubsection{
                    Skew-hermitian case
                     }

We use the setting of \S\ref{notations}. Suppose that $\V$ is skew-hermitian. Recall that $\omega=1$ or $\sgn$. In this case, $G = \C^1$ and 
\[
 U(V^\trd) = \left\{ \begin{pmatrix} 1 & 0 \\ x & 1 \end{pmatrix} \:\middle| x \in i\R \right\}.
\]
Any irreducible representation of $G$ is of the form 
\begin{align}\label{pi_l}
 \pi_l:\C^1 \rightarrow \C^\times: z \mapsto z^l
\end{align}
for some $l \in \Z$. We know that 
\[
\std{\circ}\phi_{\pi_l} \otimes \omega \cong D_{2l} \otimes \omega \cong D_{2l},
\]
so we have
\[
\gamma(s+\frac{1}{2},\phi_{\pi_l}\otimes\omega,\psi,\std) = i(-1)^l \frac{\Gamma_\C(-s+\frac{1}{2}+|l|)}{\Gamma_\C(s+\frac{1}{2}+|l|)}.
\]
It is obvious that $\gamma(s, \phi_{\pi_l}\otimes\omega,\std,\psi)$ does not depend on $\omega$. On the other hand, $\gamma^\V(s,\pi\times\omega,\psi)$ does also not depend on $\omega$.
Thus we may assume $\omega=1$. Moreover, we may assume $l \geq 0$ by the property \eqref{sd}. Then, by Proposition \ref{cal} \eqref{skherch} below, we have 
\[
\gamma(s+\frac{1}{2},\phi_{\pi_l}\otimes\omega,\psi,\std)
=\gamma^\V(s+\frac{1}{2}, \pi_l\times\omega,\psi).
\]

\begin{rem}
Let $\V$ be an anisotropic $\epsilon$-hermitian space over $\H$. Lemma \ref{eigen=gamma} holds even in the skew-hermitian case. By the lemma, we have
\begin{align}\label{eigen=gamma2}
l_{\psi_A}(f_s) = \Gamma^\V(s,\pi,1,A,\psi) l_{\psi_A}(f_{-s}).
\end{align}
for an irreducible representation $\pi$ of $G$, $f \in \tilde{\pi}\otimes\pi \subset I(s,0)$, and $A \in \u(V^\tru)_{reg}$. This equation is useful to compute the $\gamma$-factor in the case where $\V$ is anisotropic.
\end{rem}


\subsubsection{
                    A remark on the non-archimedean case
                     }\label{nonarc}

In the non-archimedean case, the local Langlands correspondence is partly proved. Although it is not completed yet in general, we can conclude the equation \eqref{desi} admitting the local Langlands correspondence.
 
Let $F$ be a non-archimedean local field of characteristic zero, let $\omega$ be a character of $F^\times$, let $D$ be a quaternion algebra over $F$ (it may be split), let $\V=(V,h)$ be an $\epsilon$-hermitian space over $D$ and let $G$ be the isometry group of $\V$. 
We may fix a globalization of $(F,D,\V,G,\omega)$ as follows:
\begin{itemize}
\item a number field $\F$ and a place $v_0$ such that $\F_{v_0} = F$;
\item a quaternion algebra $\D$ over $\F$ such that $\D_{v_0} = D$ and $\D_v$ are split for all non-archimedean places $v \not=v_0$;
\item an $\epsilon$-hermitian space $\underline{\V}$ over $\D$ and its isometry group $\underline{G}$ such that $\underline{\V}_{v_0} =\V$, $\underline{\V}_v$ is unramified and $\underline{G}(F_v)$ is quasi split over $F_v$ for all non-archimedean places $v \not=v_0$;
\item a Hecke character $\underline{\omega}$ of $\A^\times$ such that $\underline{\omega}_{v_0} = \omega$ is $\underline{\omega}_v$ are unramified for all non-archimedean places $v \not=v_0$, where $\A$ is the ring of adeles of $\F$.
\end{itemize}

We admit the following two (expected) hypotheses;
\begin{enumerate}
\item The local Langlands correspondence for $G$; \label{hyp1}
\item Existence of the global functorial lifting to $\GL_N$ associated to the standard representation of $^L\underline{G}$ into $\GL_N(\C)$ for an irreducible cuspidal automorphic  representation of $\underline{G}(\A)$. \label{hyp2}
\end{enumerate}

\begin{rem}
These hypotheses were proved by Arthur \cite{Art13} and Mok \cite{Mok15} for quasi-split classical groups. Moreover,  Kaletha, Minguez, Shin, and White extended their work to inner-forms of unitary groups \cite{KMSW14}.
\end{rem}

By a property of the local Langlands correspondence, it suffices to show the equation \eqref{desi} for all irreducible tempered representations. On the other hand, an irreducible tempered representation $\pi$ can be realized as a direct summand of a representation induced from an irreducible square integrable representation $\sigma$ of some Levi subgroup, and the $L$-parameter of $\pi$ factors through that of $\sigma$. Thus, it suffices to prove the equation \eqref{desi} for all square integrable representations. 

Let $\pi$ be an irreducible square integrable representation of $G$. Then, one can apply \cite[Theorem 5.8]{Sug12} with $S = \{v_0\}$ and $\widehat{U} = \{\pi\}$. Thus, there is an irreducible cuspidal automorphic representation $\Pi$ of $\underline{G}(\A)$ such that 
\begin{itemize}
\item $\Pi_{v_0} \cong \pi$,
\item $\Pi_v$ is unramified for all non-archimedean places $v \not= v_0$. 
\end{itemize}
Note that, by a property of functrial lifting, we have
\[
\gamma(s,\phi_{\Pi_v}\otimes\underline{\omega}_v, \std, \psi) 
 = \gamma^{GJ}(s,L(\Pi)_v\otimes\underline{\omega}_v,\psi)
\]
for all places $v$. Here, we denote by $L(\Pi)$ the global functorial lifting of $\Pi$.
Then, by the global functional equation of $\pi\times \omega$ (Theorem \ref{maingamma} \eqref{gfe}) and that of $L(\Pi)\otimes\omega$, we can  conclude the equation \eqref{desi} at $v_0$ from those at $v\not= v_0$.


\section{
            Calculations
            }\label{calcu}

In \S6.4, we prove the properties \eqref{min} and \eqref{arc} of Theorem \ref{maingamma}
admitting a formula of the $\gamma$-factor of the trivial representation (and the characters in the skew-hermitian case) of quaternionic unitary groups. In this section, we compute them to finish the proof of Theorem \ref{maingamma}.

\begin{prop}\label{cal}
Let $F$ be a local field of characteristic $0$, let $\psi$ be a non-trivial additive character of $F$, let $D$ be a quaternion algebra over $F$, and  let $\V$ be an arbitrary $n$-dimensional $\epsilon$-hermitian space over $D$. Note that $D$ is possibly split. We denote by $1$ the trivial representation of $G$.
\begin{enumerate}
\item In the hermitian case, we have \label{hertriv}
\[
\gamma^\V(s+\frac{1}{2}, 1\times1, \psi) = \prod_{j=-n}^n\gamma_F(s+\frac{1}{2}+j,1,\psi).
\]
\item In the skew-hermitian case, we have \label{skhertriv}
\[
\gamma^\V(s+\frac{1}{2}, 1\times1, \psi) = \gamma_F(s+\frac{1}{2}, \chi_{\disc(\V)}, \psi)
\prod_{j=-(n-1)}^{n-1}\gamma_F(s+\frac{1}{2}+j,1,\psi).
\]
\item In the case $F=\R, D = \H$, and $\V = (D, \langle i \rangle)$, for $A \in \u(V^\tru)_{reg}$, we have \label{skherch}
\[
\gamma^\V(s+\frac{1}{2}, \pi_l\times 1,\psi) = i(-1)^l \frac{\Gamma_\C(-s+\frac{1}{2} + l)}{\Gamma_\C(s+\frac{1}{2}+l)}
\]
where $\pi_l$ is the character of $G$ defined by \eqref{pi_l} and we fix $\psi$ as $\psi(x) = e^{2\pi i x}$ for $x \in \R$.
\end{enumerate}
\end{prop}

\begin{proof}
When $D$ is split, we know that this formulas \eqref{hertriv} and \eqref{skhertriv} hold. 
Thus, we may assume that $D$ is a division quaternion algebra over $F$. 
Once \eqref{hertriv} and \eqref{skhertriv} in the case where $F=\R, D =\H, n=0,1$ are proved,
by using multiplicativity (Theorem \ref{maingamma} \eqref{gfe}), we have \eqref{hertriv} and \eqref{skhertriv} in the case where $F=\R, D=\H$ and $\V$ has a $\lfloor n/2 \rfloor$-dimensional totally isotropic subspace. 
Then, by using the global argument as in \S\ref{uniqprf}, we have \eqref{hertriv} and \eqref{skhertriv} in the case $\V$ is anisotropic. Moreover, by using multiplicativity (Theorem \ref{maingamma} \ref{gfe}) again, we have \eqref{hertriv} and \eqref{skhertriv} in the general case.

Now we prove \eqref{hertriv}. By the above discussion, it only remains to show \ref{hertriv} in the case $F=\R, D =\H$, and $n=0,1$. If $n=0$, the claim is obvious.
Let $n=1$. We may fix $\psi$ as $\psi(x) = e^{2\pi i x}$ for $x \in \R$. We may put $\V = (D, \langle 1 \rangle)$ and  
\begin{align}\label{aaaa}
A = \begin{pmatrix} 0 & i \\ 0 & 0 \end{pmatrix} \in \u(V^\tru).
\end{align}
Take the $K$-invariant section $f_s \in I(s,\omega)$ with $f_s(e)=1$.
The equation 
\[
\begin{pmatrix} 1 & 0 \\ x & 1 \end{pmatrix}
= \begin{pmatrix} \frac{1}{\sqrt{1+xx^*}} & \frac{x^*}{\sqrt{1+xx^*}} \\
                        0 & \sqrt{1+xx^*} \end{pmatrix}
   \begin{pmatrix} \frac{1}{\sqrt{1+xx^*}} & \frac{-x^*}{\sqrt{1+xx^*}} \\
                        \frac{x}{\sqrt{1+xx^*}} & \frac{1}{\sqrt{1+xx^*}} \end{pmatrix}
\]
gives an Iwasawa decomposition of 
\[
\begin{pmatrix} 1 & 0 \\ x & 1\end{pmatrix} \in U(V^\trd)
\]
in $G^\Box$. Hence we have
\[
f_s(\begin{pmatrix} 1 & 0 \\ iy+jz+kw & 1 \end{pmatrix}) 
= \left(\frac{1}{1+y^2+z^2+w^2}\right)^{s+\frac{3}{2}}
\]
for $y,z,w  \in \R$ and then we have
\begin{align*}
&\Gamma_\C(s+\frac{3}{2}) l_{\psi_A}(f_s) \\
&=2(2\pi)^{-s-\frac{3}{2}} \int_0^\infty \int_{\R^3}
    \left(\frac{1}{1+y^2+z^2+w^2}\right)^{s+\frac{3}{2}} e^{-t}t^{s+\frac{3}{2}} e^{4\pi i y} 
        \:dy \:d^\times t \\
&=2\int_0^\infty\int_{\R^3} e^{-2\pi(1+y^2+z^2 + w^2)t}e^{4\pi i y} 
   \:dy\:dz\:dw\: t^{s+\frac{3}{2}}\:d^\times t \\
&= \frac{1}{\sqrt{2}}\int_0^\infty e^{-2\pi(t+\frac{1}{t})}t^s \:d^\times t.
\end{align*}
Note that this integral is not zero. By the change of variable $t \leftrightarrow t^{-1}$, we get
\[
\Gamma_\C(s+\frac{3}{2})l_{\psi_A}(f_s) = \Gamma_\C(-s+\frac{3}{2})l_{\psi_A}(f_{-s}).
\]
Therefore, by \eqref{eigen=gamma2}, we have 
\[
\Gamma^\V(s+\frac{1}{2},\pi\times1,\psi) 
 = \frac{\Gamma_\C(-s+\frac{3}{2})}{\Gamma_\C(s+\frac{3}{2})}.
\]
Since $\omega_s(N_V(A)) = 1, c_{\pi}(-1) = 1$ and 
\[
\gamma(s + \frac{1}{2},\sgn,\psi) = i\frac{\Gamma_\R(-s+\frac{3}{2})}{\Gamma_\R(s+\frac{3}{2})}, \ \ 
\epsilon(\frac{1}{2},\sgn,\psi) = i,
\]
we have
\begin{align*}
 \gamma^\V(s+\frac{1}{2},\pi\times 1, \psi) 
            &= \frac{\Gamma_\C(-s+\frac{3}{2})\Gamma_\R(-s+\frac{3}{2})}
                       {\Gamma_\C(s+\frac{3}{2}) \Gamma_\R(s+\frac{3}{2})} \\
           &= \frac{\Gamma_\R(-s-\frac{1}{2})
                       \Gamma_\R(-s+\frac{1}{2})\Gamma_\R(-s+\frac{3}{2})}
                        {\Gamma_\R(s-\frac{1}{2})
                        \Gamma_\R(s+\frac{1}{2})\Gamma_\R(s+\frac{3}{2})}.
\end{align*}
Hence we have \eqref{hertriv} with $n=1$, and hence we complete the proof of \eqref{hertriv}.

Now we prove \eqref{skherch}. Note that we also obtain \eqref{skhertriv} with $F=\R, n=1$ by putting $l=0$. Take $f \in \pi_{-l} \otimes \pi_{l} \subset I(s,0)$ such that $f(e) = 1$,  and take $A$ as in \eqref{aaaa}. Then, 
\begin{align}\label{skint}
\Gamma_\C(s+\frac{1}{2}+l) l_{\psi_A}(f_s) 
= \int_0^\infty e^{-2\pi(t+\frac{1}{t})}P(t)t^s \:d^\times t
\end{align}
where
\[
P(t) = t^{\frac{1}{2}+l}\int_\R e^{-2\pi t(y-\frac{i}{t})^2}(iy-1)^{2l} \:dy.
\]
By Cauchy's integral theorem, $P(t)$ becomes
\[
\int_\R e^{-2\pi y^2}(iy - (\frac{1}{\sqrt{t}} + \sqrt{t}))^{2l} \:dy,
\]
which is invariant under the permutation $t \leftrightarrow t^{-1}$. Note that the right hand side of \eqref{skint} is not zero since it is the Mellin transform of the non-zero function $e^{-2\pi(t+t^{-1})}P(t)$.
Thus, by the change of variable $t \leftrightarrow t^{-1}$ of the right hand side of \eqref{skint}, we get 
\[
\Gamma_\C(s+\frac{1}{2}+l)l_{\psi_A}(f_s) = \Gamma_\C(-s+\frac{1}{2}+l)l_{\psi_A}(f_{-s}).
\]
Then, by \eqref{eigen=gamma2}, we have \eqref{skherch}, and hence we have \eqref{skhertriv}.
\end{proof}

Proposition \ref{cal} contains the things which are not proved in \S\ref{prf} (see \S\ref{relLLC}). Hence we complete the proof of Theorem \ref{maingamma}.


\section{
           Applications
           }\label{appli}

In this section, we give two applications of the main theorem.

\subsection{
                The local root number
                }\label{appli1}

In \cite{LR05}, they determine the root number of irreducible representations for symplectic groups and (even) orthogonal groups. We can consider the analogue of their work.
In this subsection, we suppose that $\V$ is an $\epsilon$-hermitian space and $\omega^2 = 1$.

At first, we note the irreducibility of $I(0,\omega)$, which is a special case of \cite[Theorem1, 2]{Yam11}.

\begin{prop}
The degenerate principal series representation $I(0,\omega)$ is irreducible as a representation of $G^\Box$.
\end{prop}

Then, the normalized intertwining operator $M(0,\omega,A,\psi)$ acts on $I(0,\omega)$ as the  multiplication by a scalar for $A \in \u(V^\tru)_{reg}$. However, by the relation $l_{\psi_A} = M(0,\omega,A,\psi)\circ l_{\psi_A}$, the scalar is $1$. Thus, we have
\[
\epsilon^\V(\frac{1}{2}, \pi\times\omega,\psi) = \gamma^\V(\frac{1}{2}, \pi\times\omega,\psi)
= c_\pi(-1)R(0,\omega,A,\psi).
\]
By computing the correcting factor, we have the formula of the root number: 

\begin{prop}
\[
\epsilon^\V(\frac{1}{2}, \pi\times\omega,\psi)
 = c_\pi(-1) \omega(-1)^n
   \begin{cases}
    \epsilon(\frac{1}{2}, \omega, \psi) & \mbox{in the hermitian case}\\
    \omega(\disc(\V))\epsilon(\frac{1}{2}, \chi_{\disc(\V)},\psi) 
    & \mbox{in the skew-hermitian case}
    \end{cases}
\]
\end{prop}

\begin{proof}
Let $\V$ be a hermitian space. Since $\omega^2=1$, we have
\[
\gamma(\frac{1}{2}, \omega\chi_{\disc(A)}, \psi) \epsilon(\frac{1}{2},\chi_{\disc(A)},\psi)^{-1}
= \omega(\disc(A)) \epsilon(\frac{1}{2},\omega,\psi)
\]
(\cite[\S3, Corollary 2]{Tat77}). Hence, we have
\begin{align*}
\epsilon^\V(\frac{1}{2}, \pi\times\omega,\psi)
&= c_\pi(-1) \omega(N_V(A))\gamma(\frac{1}{2}, \omega\chi_{\disc(A)}, \psi) \epsilon(\frac{1}{2},\chi_{\disc(A)},\psi)^{-1} \\
&= c_\pi(-1) \omega(-1)^n \epsilon(\frac{1}{2},\omega,\psi).
\end{align*}

Let $\V$ be a skew-hermitian space. Then, $\disc(A) = \disc(\V)$ in $F^\times/F^{\times2}$.
Hence, we have
\begin{align*}
\epsilon^\V(\frac{1}{2}, \pi\times\omega,\psi)
& = c_\pi(-1)\omega(N_V(A))\epsilon(\frac{1}{2}, \chi_{\disc(\V)},\psi) \\
& = c_\pi(-1)\omega(-1)^n\omega(\disc(\V))\epsilon(\frac{1}{2}, \chi_{\disc(\V)},\psi).
\end{align*}
\end{proof}

\subsection{
                The doubling zeta integral of representations induced from a minimal parabolic subgroup
                }\label{appli2}

In this subsection, over a non-archimedean local field of odd residual characteristic, we compute the zeta integral of some spherical representations with respect to a certain subgroup.
Note that if $G$ is unramified, then the spherical representations above are the unramified representations. 

Let  $F$ a non-archimedean local field of characteristic $0$, and let $D$ be a division quaternion algebra of $F$. In this subsection, we assume that the residue characteristic of $F$ is not $2$. Let $\V$ be an $n$-dimensional $\epsilon$-hermitian space. We can take a basis $v_1, \ldots, v_n$ of $V$ such that 
\[
(h(v_i,v_j))_{ij} = \begin{pmatrix} 0 & 0 & J_r \\
                                         0 & R_0 & 0 \\
                                \epsilon J_r & 0 & 0 \end{pmatrix}
\]
with $R_0  =\diag(\alpha_1, \ldots, \alpha_{n_0}), \ord_D(\alpha_i) = 0,-1$ for $i=1, \ldots n_0, 2r+n_0 = n$, and
\[
J_r = \begin{pmatrix}  0 & & 1 \\
                              & \iddots & \\
                             1 & & 0 \end{pmatrix} \in \GL_r(D). 
\]
We choose a basis $e_1, \ldots, e_{2n}$ of $V^\Box = V \times V$ defined by 
\[
e_j = 
\begin{cases}
(v_j,v_j) & 1 \leq j \leq n \\
(v_{j-n}, -v_{j-n}) & n+1 \leq j \leq 2n.
\end{cases}
\]
We may regard $G^\Box \subset \GL_{2n}(D)$ by this basis, and we choose the maximal compact subgroup
\[
K = \{ g \in G^\Box \mid \begin{pmatrix} 1 & 0 \\ 0 & R \end{pmatrix} g \begin{pmatrix} 1 & 0 \\ 0 & R^{-1} \end{pmatrix} \in \GL_{2n}(\cO_D)\}.
\]
Let $W_i$ be a subspace of $V$ spanned by $v_1, \ldots, v_i$ for $i=1, \ldots, r$, and let $P_0$ be the stabilizer of the flag $(0 \subsetneqq W_1 \subsetneqq \cdots \subsetneqq W_r)$ of $\V$. Then, $P_0$ is a minimal parabolic subgroup of $G$ and its Levi subgroup $M_0$ is canonically isomorphic to $(D^\times)^r \times G_0$ where $G_0$ is the unitary group of the $\epsilon$-hermitian space $(D^{n_0}, \langle R_0 \rangle)$. Let $\sigma_0$ be a trivial representation of $G_0$, let $\sigma_i$ be a character of $D^\times$ defined by $\sigma_i(x) = |N_D(x)|^{s_i}$ for some $s_i \in \C$, and let $\sigma$ the representation $\otimes_{i=0}^l\sigma_i$ of $M_0$.
Put $C_0 := G \cap \GL_n(\cO_D)$, and put
\[
C_1 := \{ g \in C_0 \mid R(g-1) \in \M_n(\cO_D)\}.
\]
Then $C_1$ is an open compact subgroup of $G$.

\begin{prop}\label{appl2main}
Let $f_s^\circ \in I(s,1)^K$ be a non-zero $K$ invariant section with $f_s^\circ(e) = 1$,  $\pi$ be a constituent of $\Ind_{P_0}^G\sigma$. Suppose that $\pi$ has a non-zero $C_1$ fixed vector. If we denote by $\xi$ a bi-$C_1$ invariant matrix coefficient of $\pi$ with $\xi(e)=1$, then we have
\[
Z^{\V}(f_s^\circ,\xi) = \frac{\Vol(C_1)}{d^\V(s)} \prod_{i=0}^rL^{\V_i}(s+\frac{1}{2},\sigma_i\times1)
\]
where
\[
d^\V(s)=
\begin{cases}
\zeta_F(s+m+\frac{1}{2}) \prod_{i=1}^{\lfloor n/2 \rfloor}\zeta_F(2s+2n+1-4i) & \mbox{in the hermitian case} \\
\prod_{i=1}^{\lceil n/2\rceil}\zeta_F(2s + 2n+3-4i) & \mbox{in the skew-hermitian case}.
\end{cases}
\]
Note that if $n_0=0$, then $L^{\V_0}(s,\sigma_0)$ denotes 
\[
\begin{cases}
1 & \mbox{in the hermitian case}, \\ L_F(s,\chi_{\disc(\V)}) & \mbox{in the skew-hermitian case}.
\end{cases}
\]
\end{prop}

This proposition is proved at the end of this subsection. We start with a basic lemma:

\begin{lem}\label{basic_lemma}
For $g \in G$,  we have $|\Delta((g,1))| \leq 1$. Moreover, $|\Delta((g,1))| = 1$ if and only if $g \in C_1$.
\end{lem}

\begin{proof}
By considering the Iwasawa decomposition in $G^\Box$, we can take $a \in \GL_n(D)$ and $X \in \M_n(D)$ with ${}^t\!X^* = -\epsilon X$ such that 
\begin{align}\label{flem}
\begin{pmatrix} 1 & X \\ 0 &1 \end{pmatrix}^{-1}
\begin{pmatrix} a & 0 \\ 0 & {}^t\!a^{*-1} \end{pmatrix}^{-1}
\begin{pmatrix} 1 & 0 \\ 0 & R \end{pmatrix} 
(g,1) \begin{pmatrix} 1 & 0 \\ 0 & R^{-1} \end{pmatrix} \in \GL_{2n}(\cO_D).
\end{align}
Consider two submodules 
\[
L_1 = R(g-1)\cdot \cO_D^n, \  L_2 = R(g+1)R^{-1} \cdot \cO_D^n 
\]
of the vector space $D^n$, then \eqref{flem} concludes that 
\[
L_1 + L_2 = {}^t\!a^{*-1}\cO_D^n.
\]
On the other hand, considering in ${}^t\!g^{-1} = RgR^{-1}$, we have
\[
L_1 = ({}^t\!g^{-1}-1)R\cdot \cO_D^n \supset ({}^t\!g^{-1}-1)\cdot \cO_D^n, 
\ L_2 = ({}^t\!g^{-1}+1)\cdot\cO_D^n,
\]
and thus $L_1 + L_2 \supset \cO_D^n$. Hence we have
\[
|\Delta((g,1))| = |N(a)| = |N({}^t\!a^{*-1})|^{-1} \leq 1.
\]
Moreover, $|\Delta((g,1))| = 1$ if and only if $L_1 \subset \cO_D^n$ and $L_2 \subset \cO_D^n$, which is equivalent to the condition of the lemma.
\end{proof}

Consider the partition of the integral
\[
Z^\V(f_s^\circ,\xi) = \int_{C_1}\xi(g) \: dg + \int_{G-C_1}f_s^\circ((g,1))\xi(g) \: dg.
\]
If $s_0$ be a sufficiently large real number so that $Z^\V(f_{s_0}^\circ,\xi)$ converges absolutely, then, by Lemma \ref{basic_lemma}, we have
\begin{align*}
\left| \int_{G-C_1}f_s^\circ((g,1))\xi(g) \: dg \right| 
&\leq \int_{G-C_1}|\Delta((g,1))|^{s-s_0}|f_{s_0}^\circ((g,1))\xi(g)|\: dg \\
&\leq q^{-(\Re s-s_0)}\int_G |f_{s_0}^\circ((g,1))\xi(g)|\: dg
\end{align*}
for $\Re s > s_0$. Thus we have
\begin{align}\label{limitation}
\lim_{\Re s \rightarrow \infty}Z^\V(f_s^\circ,\xi) = \Vol(C_1).
\end{align}

Now we prove Proposition \ref{appl2main}. Note that $Z^\V(f_s^\circ, \xi) \not=0$ by \eqref{limitation}. Put $\Xi(q^{-s})$ by 
\[
\frac{Z^\V(f_s^\circ,\xi)}{\prod_{i=0}^rL^{\V_i}(s + \frac{1}{2}, \sigma_i\times1)}.
\]
The ``g.c.d property'' \cite[Theorem 5.2]{Yam14} and \cite[Lemma 6.1]{Yam14} conclude that $\Xi(q^{-s})$ is a polynomial in $q^{-s}$ and $q^s$. Moreover, \eqref{limitation} implies that it is a polynomial in $q^{-s}$ with the constant term $\Vol(C_1)$.
We define the polynomial $D(q^{-s})$ by $d^\V(s)^{-1}$. 
Since the action of the normalized intertwining operator on $f_s^\circ$ is given by
\[
M^*(s,\omega,A,\psi)f_s^\circ = e(G) q^{-n's} \cdot \frac{D(q^{-s})}{D(q^s)}f_{-s}^\circ
\]
(Proposition \ref{norm_ex} and \cite[Proposition 3.5]{Shi99} (or \cite[Theorem 3.1]{Cas80})), 
and since the $\gamma$-factor of $\pi$ is given by 
\[
\gamma^\V(s+\frac{1}{2}, \pi\times1,\psi)
= e(G) q^{-n's} \prod_{i=0}^r\frac{L^{\V_i}(-s + \frac{1}{2}, \sigma_i^\vee\times1)}{L^{\V_i}(s + \frac{1}{2}, \sigma_i\times1)}
\]
where
\[
n' = \begin{cases} 2\lceil \frac{n}{2} \rceil & \mbox{in the hermitian case}, \\
                         2\lfloor \frac{n}{2}  \rfloor & \mbox{in the skew-hermitian case}. \end{cases}
\]
(Proposition \ref{cal}), we can rewrite the functional equation \eqref{locFE_zeta} as
\[
\Xi(q^{-s})D(q^s) = \Xi(q^s) D(q^{-s}).
\]
However, for sufficiently large $m$, $q^{-ms}D(q^{s})$ is coprime to $D(q^{-s})$ as polynomials. Thus, comparing the constant term of $D(q^{-s})$ and $\Xi(q^{-s})$, we have 
\[
\Xi(q^{-s}) = \Vol(C_1) \cdot D(q^{-s}).
\]
Hence we have the proposition.


\end{document}